\documentclass{amsart}

\usepackage[latin1]{inputenc}
\usepackage[T1]{fontenc}
\usepackage{enumerate}

\usepackage[english]{babel}

\usepackage{amsmath,amsfonts,amssymb,amsthm,amscd}

\newtheorem{theorem}{Theorem}[section]
\newtheorem{lemma}[theorem]{Lemma}
\newtheorem{corollary}[theorem]{Corollary}
\newtheorem{proposition}[theorem]{Proposition}
\newtheorem{claim}{Claim}%[theorem]

\theoremstyle{definition}
\newtheorem{definition}[theorem]{Definition}
\newtheorem{example}[theorem]{Example}

\newtheorem{question}{Question}

\theoremstyle{remark}
\newtheorem{remark}[theorem]{Remark}

\numberwithin{equation}{section}

\newcommand{\N}{\mathbb N}
\newcommand{\K}{\mathbb K}
\newcommand{\R}{\mathbb R}

\newcommand{\C}{\mathbb C}
\newcommand{\D}{\mathbb D}

\title{Algebrable sets of hypercyclic vectors for convolution operators}
\author[J.\ B\`{e}s and D.\ Papathanasiou]{J.\ B\`{e}s and D.\ Papathanasiou}
\address{J. B\`{e}s, Department of Mathematics and Statistics,
Bowling Green State University,
Bowling Green, Ohio 43403,
USA}
\email{jbes@bgsu.edu}
\address{D. Papathanasiou, Universit\'e Clermont Auvergne, CNRS, LMBP, F-63000 Clermont Ferrand, France.}
\email{dpapath@bgsu.edu}
\thanks{This work is supported in part by MEC, Project
MTM 2016-7963-P. 
We also thank Fedor Nazarov and an anonymous referee for key observations for Section~\ref{SS:1.2}.
}
\date{March 4, 2019}

\allowdisplaybreaks[3]
\begin{document}
\begin{abstract}
We show that several convolution operators on the space of entire functions, such as the MacLane operator, support a dense hypercyclic algebra that is not finitely generated. Birkhoff's operator also has this property on the space of complex-valued smooth functions on the real line.
\end{abstract}
\maketitle
{\large 

\section{Introduction}
The search for large algebraic structures (e.g., linear spaces, closed subspaces, or infinitely generated algebras) in non-linear settings has drawn increasing interest over the past decade \cite{aron_bernal-gonzalez_pellegrino_seoane-sepulveda2015lineability,bernal-gonz\'alez_pellegrino_seoane}.
One such setting is given by the set
\[
HC(T)=\{ f\in X: \ \ \{ f, Tf, T^2f,\dots \} \mbox{ is dense in $X$} \}
\]
of {\em hypercyclic vectors} for a given operator $T$ on a topological vector space $X$.  
Whenever $HC(T)$ is non-empty we say that $T$ is a {\em hypercyclic operator}, and in this case 
 $HC(T)\cup\{ 0\}$ always contains a dense linear subspace that is $T$-invariant, see \cite{wengenroth2003hypercyclic}. 
The question whether  $HC(T)\cup\{0\}$ contains a closed and infinite dimensional subspace $M$ (which in turn is called a {\em hypercyclic subspace}) has a negative answer for some hypercyclic operators $T$ and a positive one for others 
\cite{bernal-gonzalez_montes-rodriguez1995non-finite-dimensional,LeM01,gonzalez_leon-saavedra_montes-rodriguez2000semi-fredholm}, and many findings in this direction have been made; see
e.g. \cite[Ch.\ 8]{bayart_matheron2009dynamics} and \cite[Ch.\ 10]{grosse-erdmann_peris-manguillot2011linear}, and the more recent work \cite{bonilla_grosse-erdmann2012frequently,Menet1bis,Menet2,Menet1,bes_menet2015existence,bayart_ernst_menet}%%%%XXX

The search for algebras of hypercyclic vectors, other perhaps than Read's \cite{Re88} construction of an operator on $\ell_1(\N)$ (endowed with any algebra structure) for which every non-zero vector is hypercyclic, may be traced back to the work of
Aron et al \cite{aron_conejero_peris_seoane-sepulveda2007powers,aron_conejero_peris_seoane-sepulveda2007sums} who showed that no translation operator can support a hypercyclic algebra on the space $H(\C)$ of entire functions, endowed with the compact open topology. They also showed that, in
 sharp contrast with the translations operators,  the collection of entire functions $f$ for which every power $f^n$ $(n=1,2,\dots )$ is hypercyclic for the operator $D$ of complex differentiation 
is residual in $H(\mathbb{C})$. 

A few years later Shkarin \cite[Thm.\ 4.1]{shkarin2010on} showed that $HC(D)$ contains both a hypercyclic subspace and a hypercyclic algebra, and with a different approach Bayart and Matheron
\cite[Thm.\ 8.26]{bayart_matheron2009dynamics} also showed that the set of $f\in H(\mathbb{C})$ that  generate an algebra consisting entirely (but the origin) of hypercyclic vectors for $D$ is residual in $H(\mathbb{C})$. This prompted the following question, which also appeared in \cite[p.~105]{bernal-gonz\'alez_pellegrino_seoane}, 
\cite[p.~185]{aron_bernal-gonzalez_pellegrino_seoane-sepulveda2015lineability}:

\begin{question} \label{Q:I1} ({\bf Aron}~\cite[p.~217]{bayart_matheron2009dynamics})   

\begin{quote}
Is it possible to find a hypercyclic algebra for $D$ that is not finitely generated? 
\end{quote}
\end{question}

By imitating Bayart and Matheron's proof,  Conejero and the authors \cite{bes_conejero_papathanasiou2016convolution} showed that $P(D)$ supports a hypercyclic algebra whenever $P$ is a non-constant polynomial vanishing at zero, and with a new approach they obtained hypercyclic algebras for convolution operators not induced by polynomials, such as $\sin(aD)$ or $\cos(aD)$, where $0\ne a\in\C$, see \cite{BCP2}.  
But each of the hypercyclic algebras obtained on $H(\C)$ so far is generated by a single entire function and hence contained in a closed hyperplane. Moreover, each subalgebra of a singly generated algebra is finitely generated \cite{gale}, so larger algebras are required to answer Question~\ref{Q:I1}.

We recall that as established by Godefroy and Shapiro \cite{godefroy_shapiro1991operators}, convolution operators on $H(\C)$ are precisely those operators of the form
\[
f\overset{\Phi(D)}{\mapsto} \sum_{n=0}^\infty a_n D^nf   \ \ \ (f\in H(\C))
\]  
where $\Phi(z)=\sum_{n=0}^\infty a_n z^n \in H(\C)$ is of exponential type (i.e., $|a_n|\le M\ \frac{R^n}{n!} (n=0, 1,\dots )$, for some $M, R>0$), and that each such operator $\Phi(D)$ is hypercyclic whenever $\Phi$ is non-constant.

In this paper we answer Question~\ref{Q:I1} in the affirmative and show that a large class of convolution operators on $H(\C)$ supports dense hypercyclic algebras that are not finitely generated. 
We recall that a subset $M$ of a Fr\'echet algebra $X$ is said to be {\em algebrable} provided $M\cup\{ 0\}$ contains a subalgebra of $X$ that is not finitely generated and {\em dense algebrable} if such a subalgebra is dense in $X$ \cite{aron_bernal-gonzalez_pellegrino_seoane-sepulveda2015lineability}. So Question~\ref{Q:I1} asks whether the set of hypercyclic vectors for $D$ is algebrable. We show that this set is indeed {\em strongly algebrable} in the sense of Bartoszewicz and G{\l}{\c a}b \cite{barto},
see Definition~\ref{D:strongly}.

\begin{theorem} \label{T:1}
Let $\Phi\in H(\C )$ be %an entire function
 of finite exponential  type with $|\Phi (0)|<1$ and so that the level set $\{ z: \ |\Phi(z)|=1 \}$ contains a non-trivial, strictly convex compact arc $\Gamma$ satisfying
\begin{equation} \label{eq:T:1} 
\mbox{conv}(\Gamma\cup \{ 0 \} ) \setminus \Gamma \subset \Phi^{-1} (\mathbb{D}).
\end{equation}
Then the set of of hypercyclic vectors for the convolution operator $\Phi(D)$ is dense algebrable.
Moreover, 
\begin{enumerate}
\item[(a)]\  For each $N\in\N$ the set of $f=(f_j)_{j=1}^N$ in $H(\C)^N$ that freely generate a hypercyclic algebra for $\Phi(D)$ is residual in $H(\C)^N$, and
\item[(b)]\ The set of $f=(f_j)_{j=1}^\infty\in H(\C)^\N$ that freely generate a dense hypercyclic algebra for $\Phi(D)$ is residual in $H(\C)^\N$. In particular, the set of hypercyclic vectors for $\Phi(D)$ is densely strongly-algebrable.
\end{enumerate}
%\begin{enumerate}
%\item[(a)]\  For each $N\in\N$ the set of $f=(f_j)_{j=1}^N$ in $H(\C)^N$ that generate a hypercyclic algebra for $\Phi(D)$ that is not contained in a hypercyclic algebra for $\Phi(D)$ induced by $k$ generators with $k<N$ is residual in $H(\C)^N$, and
%\item[(b)]\ The set of $f=(f_j)_{j=1}^\infty\in H(\C)^\N$ that generate a dense hypercyclic algebra for $\Phi(D)$ that is not contained in a finitely generated hypercyclic algebra for $\Phi(D)$ is residual in $H(\C)^\N$.
%\end{enumerate}

\end{theorem}

Here for any  subset $A$ of the complex plane the symbol $\mbox{conv}(A)$ denotes its convex hull. Also, an arc
 $\mathcal{C}$ is said to be {\em strictly convex} provided each segment with both endpoints in $\mathcal{C}$ only intersects $\mathcal{C}$ at such endpoints.
\begin{remark}
While the geometric assumption $\eqref{eq:T:1}$  for $\Phi$ in Theorem~\ref{T:1} may be relaxed -see Remark~\ref{R:condition}-, it is stronger than the one required 
in \cite[Theorem~3]{BCP2} which established the existence of singly generated algebras and imposed no requirement for $\Phi (0)$ other than belonging to the closed unit disc.
\end{remark}

\begin{corollary} \label{C:D}
The set of hypercyclic vectors for the operator of complex differentiation is densely strongly algebrable on $H(\C)$.  
\end{corollary}

\begin{corollary}
The set of common hypercyclic vectors for $\{ \sin(kD) \}_{k\in\N}$ is densely strongly algebrable on $H(\C )$.
In general, any countable family of hypercyclic convolution operators $\{ \Phi_n(D)\}_{n\in\N}$ will have a densely strongly algebrable set of common hypercyclic vectors  if for each $n\in\N$  the  entire function $\Phi_n$ maps the origin into the open unit disc $\D$ and supports a non-trivial strictly convex arc $\Gamma_n\subset \Phi_n^{-1}(\partial\mathbb{D})$  so that
\[
\mbox{conv}(\Gamma_n\cup \{ 0 \} ) \setminus \Gamma_n \subset \Phi_n^{-1} (\mathbb{D}).
\]
\end{corollary}

We may derive from Theorem~\ref{T:1} examples of strongly algebrable sets of hypercyclic vectors for differentiation operators on the space  $C^\infty(\R, \C)$ of complex-valued of smooth functions on the real line, whose topology is given by the sequence $(p_n)$ of seminorms \[
p_n(f) =\mbox{sup}_{|x|\le n} \mbox{max}_{0\le j\le n} |f^{(j)}(x)|, \ \ \mbox{ $f\in  C^\infty(\R, \C)$.}\]
As Godefroy and Shapiro~\cite{godefroy_shapiro1991operators} observed, the restriction operator
\[
\begin{aligned}
\mathcal{R}:&H(\C)\to C^\infty(\R, \C) \\
f&=f(z)\mapsto f(x)
\end{aligned}
\]
is continuous, injective, of dense range, and multiplicative, and for any complex polynomial $P=P(z)$ %and translation operator $\tau_a$  $(a\in \R)$ 
we have
\[
\mathcal{R} P(\frac{d}{dz})  = P(\frac{d}{dx}) \mathcal{R} % \ \ \mbox{ and } \ \  \mathcal{R} \tau_a =  {\tau_a^1} \mathcal{R},
\]
Hence $\mathcal{R}$ is an algebra isomorphism onto its range and it sends hypercyclic vectors of $P(D)$ to hypercyclic vectors of $P(\frac{d}{dx})$. Thus, an infinitely generated, dense hypercyclic algebra for $P(D)$ will be maped by $\mathcal{R}$ to an infinitely generated, dense hypercyclic algebra for $P(\frac{d}{dx})$. Thus by Theorem~\ref{T:1} we have the following corollary.

\begin{corollary} \label{C:I1}
Let $P$ be a non-constant polynomial whose level set $\{ z:\ |P(z)|=1 \}$ contains a non-trivial, strictly convex compact arc $\Gamma$ so that
\[
\mbox{conv}(\Gamma\cup \{ 0 \} ) \setminus \Gamma \subset P^{-1} (\mathbb{D}).
\]
Then $P(\frac{d}{dx})$ supports a dense hypercyclic algebra on $C^\infty(\R,\C)$ that is not finitely generated.
\end{corollary}
%So Corollary~\ref{C:I1} and \cite[Corollary~9]{BCP2} give the following example. 
%\begin{example}
%Let $P(z)=a_0 + a_1z$, where $|a_0|<1$ and $0\ne a_1$.  Then $P(\frac{d}{dx})$ supports a dense hypercyclic algebra on $C^\infty(\R,\C)$ that is not finitely generated. 
%\end{example}

%%%%%%%%%%%%%%%%%%%%%%

Conejero and the authors \cite{BCP2} showed that non-trivial translations acting on $C^\infty(\R, \C)$ do support hypercyclic algebras, unlike the case when they act on $H(\C)$. We show in Corollary~\ref{C:T_a}
below that they indeed support dense hypercyclic algebras that are not finitely generated. More generally, 
we have the following Zero-One law (cf.\ Definition~\ref{densegenerator}):

\begin{theorem} \label{T:dos}
Let $X$ be a separable, commutative  $F$-algebra over the real or complex scalar field $\mathbb{K}$, and supporting a dense freely generated subalgebra. Then for any operator $T$ on $X$ that is both weakly mixing and multiplicative, the following are equivalent:
\begin{enumerate}
\item[{\rm ($a$)}]\  The operator $T$ supports a hypercyclic algebra.

\item[{\rm ($b$)}]\ For each non-zero polynomial $P$ with coefficients in $\mathbb{K}$ and in one variable,  the map $\widehat{P}:X\to X$, $f\mapsto P(f)$, has dense range.
\item[{\rm ($c$)}]\ For each $N\ge 2$ and each non-zero polynomial $P$ with coefficients in $\mathbb{K}$ and in $N$ variables, the map $\widehat{P}:X^N\to X$, $f\mapsto P(f)$, has dense range.

\item[{\rm ($d$)}]\ The set of hypercyclic vectors for $T$ is densely strongly-algebrable. Moreover, the set of $f\in X^\N$ that freely generate a dense hypercyclic algebra {\rm (}but zero{\rm )} is residual in $X^\N$.

\end{enumerate}
In particular, either {\em every} weakly mixing multiplicative operator on $X$ supports a dense hypercyclic algebra that is not finitely generated or else no such operator supports a hypercyclic algebra.

\end{theorem}

We list  in Corollary~\ref{C:T_a} below some examples showcasing Theorem~\ref{T:dos}. Cases $(ii)$ and $(iii)$ were first noted in \cite{aron_conejero_peris_seoane-sepulveda2007powers, BCP2}.
\begin{corollary} \label{C:T_a}
\

\begin{enumerate}
\item[{\rm (i)}]\
Each multiplicative weakly mixing operator on the F-algebra $C^\infty(\mathbb{R}, \C)$ over $\mathbb{K}=\C$ supports a dense, non-finitely generated algebra of hypercyclic vectors. In particular, each
translation operator $T_a$ $(0\ne a\in \mathbb{R})$ on $C^\infty(\mathbb{R}, \C)$ has this property.

\item[{\rm (ii)}]\ No multiplicative operator on the F-algebra $H(\Omega )$ over $\mathbb{K}=\C$ supports a hypercyclic algebra, where $\Omega \subset \mathbb{C}$ is simply connected.

\item[{\rm (iii)}]\ No multiplicative operator on the F-algebra $C^\infty(\mathbb{R}, \mathbb{R})$ over $\mathbb{K}=\mathbb{R}$ supports a hypercyclic algebra.

\item[{\rm (iv)}]\ No multiplicative operator on the F-algebra $H_\mathbb{R}(\C )=\{ f\in H(\C ):\ f(\mathbb{R})\subset \mathbb{R} \}$ over $\mathbb{K}=\mathbb{R}$ supports a hypercyclic algebra.

\end{enumerate}
\end{corollary}

\begin{proof}[Proof of Corollary~\ref{C:T_a}]
Notice that the set of complex polynomials on a real variable is dense in $C^\infty(\mathbb{R}, \C)$, and that the inclusion $g:\mathbb{R}\to \mathbb{C}$, $g(t)=t$, freely generates a dense subalgebra of the $F$-algebra $(X, \mathbb{K})=(C^\infty(\mathbb{R}, \C), \mathbb{R})$. Since there exists a multiplicative operator on $X$ supporting a hypercyclic algebra  \cite[Corollary~20]{BCP2}, then $(i)$ follows by Theorem~\ref{T:dos}. 
The pending conclusions $(ii)$-$(iii)$ follow by considering the polynomial $p(t)=t^2$ on Theorem~\ref{T:dos}(c).
Notice that the identity function on $\C$ freely generates a dense subalgebra of $(X, \mathbb{K})=(H(\Omega ),\mathbb{C})$, since $\Omega$ is simply connected, and also of $(X, \mathbb{K})=(H_\mathbb{R}(\mathbb{C}), \mathbb{R})$, since the set of real polynomials on a complex variable is dense on $H_\mathbb{R}(\mathbb{C})$.
Similarly, the identity function on $\mathbb{R}$ freely generates a dense subalgebra of  $(X, \mathbb{K})=(C^\infty(\mathbb{R},\mathbb{R}), \mathbb{R})$. 
\end{proof} % of proof of Corollary~\ref{C:T_a}]

\begin{remark}
We note that each translation $\tau_a$ $(0\ne a\in\R)$, as well as each convolution operator of the form $P(D)$, where $P$ is a non-constant polynomial with real coefficients with $\inf_{t\in\mathbb{R}}|P(t)|<1$, is mixing on $H_\mathbb{R}(\C )$. This follows by the Kitai Criterion and the fact that for each subset $U$ of the real line with a limit point the set of eigenvectors $\{ e^{uz}: \ u\in U \}$ for $T=\tau_a$ or $P(D)$ has dense $\mathbb{R}$-linear span in $H_\mathbb{R}(\mathbb{C})$.
\end{remark}

In Section 2 we recall the notion of strong algebrability together with basic notation and facts about algebras and generators. In particular, we indicate a Zero-One law for any separable commutative $F$-algebra $X$ that is used to show the main results: Either $X$ supports no freely-generated dense subalgebra or its generic sequence freely-generates a dense subalgebra, see Proposition~\ref{P:00}.  The proofs of Theorem~\ref{T:1} and of Theorem~\ref{T:dos} are shown in Section 3 and Section 4, respectively.

\section{Freely generated algebras and strong algebrability}  \label{SS:1.2}
Throughout this  paper the symbol $X$ denotes an $F$-algebra (i.e., a metrizable and complete topological algebra) over the 
real or complex scalar field $\mathbb{K}$, which is both separable and commutative. We respectively denote by $\mathbb{R}$, $\mathbb{C}$, $\mathbb{N}$ and $\mathbb{N}_0$ the sets of real scalars, complex scalars, positive integers, and non-negative integers.
By  $\mathbb{K} [t]$ and $\mathbb{K} [t_1,\dots ,t_N]$ $(N\in\N)$ we denote the algebras of polynomials with coefficients in $\K$ in one and $N$ variables, respectively. %, and for  $d\ge 1$ we let $\mathbb{K}_d[t]$ and $\mathbb{K}_d[t_1,\dots ,t_N]$ denote the corresponding subalgebras of polynomials of degree not exceeding $d$.
A subset $Y$ of $X$ is {\em algebraically dependent} provided there exist pairwise distinct $y_1,\dots, y_n$ in $Y$ and a non-zero polynomial 
$P\in \mathbb{K}[t_1,\dots, t_n]$ with $P(0)=0$ so that $P(y_1,\dots ,y_n)=0$.  We say that $Y$ is {\em algebraically independent } provided it is not algebraically dependent.
The subalgebra generated by  $Y$ is 
\[
\begin{aligned}
A(Y)&=\cap {\{ \mathcal{A}\supset Y:  \mbox{ $\mathcal{A}$ subalgebra of $X$} \}} \\
&= \{ P(y):   y\in Y^n, P\in \mathbb{K}[t_1,\dots, t_n]:P(0)=0, n\in\N \} 
\end{aligned}
\]
and (if $X$ has a unit) the unital algebra generated by $Y$ is 
\[
\begin{aligned}
A_1(Y)&=\cap {\{ \mathcal{A}\supset Y:  \mbox{ $\mathcal{A}$ unital subalgebra of $X$} \}} \\
&= \{ P(y):   y\in Y^n, P\in \mathbb{K}[t_1,\dots, t_n], n\in\N \} .
\end{aligned}
\]
We call $Y$ a set of generators of  the algebra $A(Y)$ and  of the unital algebra $A_1(Y)$ (if $X$ has a unit). 
\vspace{.1in}

\noindent
It is possible for an algebra generated by just two elements to support a subalgebra that is {\em not} finitely generated. Hence when searching for large algebras of hypercyclic vectors we prefer Bartoszewicz and G{\l}{\c a}b's notion of {\em strong algebrability}, formulated in terms of freely generated algebras, over the usual notion of algebrability.

\begin{definition}  \label{densegenerator}
For a cardinal $\kappa$, we say that $A$ is a {\em $\kappa$-generated free-algebra} provided there exists a subset $Y=\{ y_\gamma:   \gamma < \kappa\}$ of $A$ so that every function $f: Y\to A'$ with $A'$ an algebra can be  uniquely extended to an algebra homomorphism $\overline{f}:A\to A'$. 
The set $Y$ is called  a {\em set of free generators} of the algebra $A$. 
\end{definition}
\begin{remark} \label{R:12}
%A freely generated algebra cannot contain zero divisors, and a subset $Y$ of  $X$ generates a free algebra if and only if $Y$ is algebraically independent.   Also, a subset  $Y=\{ y_\gamma:   \gamma < \kappa\}$ of a {\em commutative} algebra $X$ generates a free subalgebra of $X$ if and only if for each finite set of pairwise distinct elements $y_{\gamma_1},\dots , y_{\gamma_n}$ of $Y$ and each polynomial $P\in \mathbb{K}[t_1,\dots ,t_n]$ with $P(0)=0$ we have
%$
%P(y_{\gamma_1},\dots , y_{\gamma_n})=0 \  \mbox{ if and only if }    \  P=0.
%$
%We also have:
A subset $Y$ of  $X$ generates a free algebra if and only if $Y$ is algebraically independent. Indeed, we have:
\begin{itemize}
\item[(a)] \ 
When $X$ has no unit,  $Y=\{ y_\gamma:   \gamma < \kappa\} \subset X$ is a set of free generators of a subalgebra $A$ of $X$ if and only if the set $\widetilde{Y}$ of elements of the form
\[
y_{\gamma_1}^{\alpha_1} y_{\gamma_2}^{\alpha_2}\cdots y_{\gamma_n}^{\alpha_n}
\]
with $y_{\gamma_1},\dots , y_{\gamma_n}$ in $Y$ pairwise distinct and with $\alpha_1,\dots , \alpha_n \in \mathbb{N}_0$ and not simultaneously zero is linearly independent and 
\[
A=A(Y)=\mbox{span}(\widetilde{Y}).
\]
\item[(b)] \ When $X$ has a unit $e$,  then $Y=\{ y_\gamma:   \gamma < \kappa\} \subset X$ is a set of free generators of a unital subalgebra $A_1$ of $X$ if and only if the set $\widetilde{Y}$ defined as in $(1)$ satisfies both that
$\widetilde{Y}\cup \{ e \}$ is linearly independent and $\widetilde{Y}\cup \{ e \}$ is a spanning set of $A_1$. %$A_1=\mbox{span}(\widetilde{Y}\cup \{ e \})$.

%\item \ A free algebra has no divisors of zero.

%\item \ Any two sets of free generators of a free algebra must have the same cardinality.  {\bf ??? Check! }. However, if {x, y} are algebraically independent over \K, then the two-freely generated algebra K[x,y] contains the subalgebra generated by \{ xy^n : n\in\N \} which is not free  nor finitely generated}

\end{itemize}
\end{remark}

\begin{definition} (Bartoszewicz and G{\l}{\c a}b \cite{barto}) \label{D:strongly}
Let $X$ be a commutative $F$-algebra, and let $\kappa$ be an infinite cardinal. A subset $E$ of $X$ is {\em strongly $\kappa$-algebrable} provided $E\cup\{ 0 \}$ contains a $\kappa$-generated free subalgebra of $X$. The set $E$ is {\em densely strongly $\kappa$-algebrable} provided  $E\cup\{ 0 \}$ contains a $\kappa$-generated dense free subalgebra of $X$. Finally, we say that $E$ is {\em strongly algebrable} (respectively, {\em densely strongly algebrable}), if it is  strongly $\kappa$-algebrable (respectively, densely strongly $\kappa$-algebrable) for some infinite cardinal $\kappa$.
\end{definition}

We use Proposition~\ref{P:00} below to show Theorem~\ref{T:1} and Theorem~\ref{T:dos}; we have been unable to find a reference and provide a proof for the sake of completeness.
\begin{proposition} \label{P:00}\
Let $X$ be a separable commutative $F$-algebra supporting a dense freely generated subalgebra.
Then we have:
\begin{enumerate}
\item[1.]\  For each $N\in\N$, the set of algebraically independent $N$-tuples of elements of $X$ is residual in $X^N$. In particular, the set of $f$ in $X^N$ that generate an algebra $A(f)$ that is not contained in any subalgebra of $X$ induced by fewer than $N$ generators is residual in $X^N$.

\item[2.]\  The set of algebraically independent sequences in $X$ that induce a dense subalgebra of $X$ is residual in $X^\N$. In particular, the set
%B
 of  $f=(f_j)$ in  $X^\N$ whose induced algebra $A(f)$ is not contained in a finitely generated subalgebra of $X$ is residual in $X^\N$.

\end{enumerate}

\end{proposition}

\begin{remark}
Natural examples of separable commutative $F$-algebras that support a dense freely-generated subalgebra include function algebras such as algebras of holomorphic (respectively, smooth) functions of finitely many variables, the disc algebra, etc. Another natural example is the algebra $\ell_1$ of absolutely convergent sequences with respect to the convolution product. Here 
 $e_0=(1,0,\dots)$ is the unit and  the singleton $\{ e_1\} =\{ (0,1,0,\dots )\}$ freely generates a dense subalgebra.  Indeed, recall that the  {\em convolution } product $x\ast y$ of two elements $x,y\in \ell_1$ is given by
\[
(x\ast y)(n):= \sum_{k=0}^n  x(k)y(n-k) \ \ \ \ (n=0,1,\dots ).
\]
So $e_1^{n}=e_n$ for each $n\in\mathbb{N}_0$ and for any polynomial $P=\sum_{j=0}^r a_j t^j$ we have 
\[
P(e_1)=\sum_{j=0}^r a_j e_1^j =\sum_{j=0}^r a_j e_j,
\]
and  $
A_1(e_1)=\{ P(e_1) :\ P\in \mathbb{K}[t] \}
$ is dense in $\ell_1$. That is, $\{ e_1^n \}_{n\in\N}\cup \{ e_0 \}$ is linearly independent and it spans a dense subspace of $\ell_1$.
\end{remark}

%We note that it is possible for a commutative $F$-algebra to support free elements and yet not support a dense freely generated algebra; simply consider $X:=\mathbb{K}\times \ell_1$ with the product $(a,x)\cdot (a',x):=( 0, x\ast x')$. 

The remaining of the section is devoted to show Proposition~\ref{P:00}.
We first note that an $N$-tuple of polynomials in $k$ variables must be algebraically dependent if $N>k$.
For each $\alpha=(\alpha_1,\dots, \alpha_N) \in \mathbb{C}^N$ we let $|\alpha|:=\sum_{j=1}^N |\alpha_j|$ and $|\alpha|_\infty :=\mbox{max}_{1\le j\le N}{|\alpha_j|}$.

\begin{proposition} \label{P:4}
Let $P_1,\dots ,P_N$ in $\mathbb{K} [t_1,\dots ,t_k]$ with $N>k$. Then there exists a nonzero polynomial $R\in \mathbb{K} [t_1,\dots ,t_N]$ with $R(0)=0$ such that $R(P_1,\dots ,P_N)=0$. 
\end{proposition}
\begin{proof}
Let $d=\max \{\deg P_1,\dots ,\deg P_N\}$, and choose $q\in \N$ such that 
\begin{equation}\label{eq:P4}
q^N>(Ndq+1)^k.\end{equation}
The vector space 
\[
W:=\mbox{span}\{ t^\beta=t_1^{\beta_1}\cdots t_k^{\beta_k}: \ \beta\in\N_0^k \mbox{ with }  | \beta |_\infty \le Ndq \}
\]
has dimension $\mbox{dim}(W)=(Ndq+1)^k$ and contains the indexed set
\[
E:=\{ \prod_{i=1}^N P_i^{\alpha_i} \}_{\{ \alpha\in\N_0^N:  1\le |\alpha|_\infty\le q \}}
\]
which by $\eqref{eq:P4}$  must be linearly dependent. So there exists a non-trivial linear combination
\[
\sum_{ \{ \alpha\in\N_0^N: \ 1\le |\alpha|_\infty \le q \} } c_\alpha \prod_{i=1}^N P_i^{\alpha_i}=0,
\]
and the conclusion holds for $R:= \sum_{\{ \alpha\in\N_0^N:\ 1\le |\alpha|_\infty\le q\}} c_\alpha \prod_{i=1}^N t_i^{\alpha_i}$.  
\end{proof}
\begin{corollary}  \label{C:05}
Let $f=(f_1,\dots, f_N)\in X^N$ so that $A(f)\subset A(h)$ 
%(or so that $A(f)\subset A_1(h)$ if $X$ has a unit) 
for some $h\in X^k$ with $k<N$. Then  $\{ f_1,\dots, f_N\}$ is algebraically dependent.   
\end{corollary}
\begin{proof}
By the assumption there exist $P_1,\dots ,P_N \in \mathbb{K}[t_1,\dots ,t_k]$ such that \[
P_i(h_1,\dots ,h_k)=f_i \ \ \ (1\leq i\leq N).\] Since $N>k$, by Proposition~\ref{P:4} there exists $0\ne R\in \mathbb{K}[t_1,\dots ,t_N]$ with $R(0)=0$ such that $R(P_1,\dots ,P_N)=0$. In particular,
 \[    R(f_1,\dots ,f_N)            =R(P_1(h_1,\dots ,h_k),\dots ,P_N(h_1,\dots ,h_k))=0.\]
\end{proof}

We note the following elementary topological fact which we use to show Proposition~\ref{P:07} below.
\begin{remark} \label{R:00}
If $K, Y, Z$ are Hausdorff topological spaces with $K$ compact and if $g:K\times Y\rightarrow Z$ is a continuous map, then the set
\[
\{x\in Y: \exists k\in K \,\, \mbox{so that} \,\, g(k,x)\in A\}
\]
is closed for each closed subset $A$ of $Z$.
\end{remark}
\begin{proposition} \label{P:07}
Let $X$ be a separable commutative  $F$-algebra that contains a freely generated dense subalgebra and let $N\ge 2$. Then the set 
$F$ of $N$-tuples $f=(f_1,\dots, f_N)\in X^N$ that are  algebraically dependent is of first category in $X^N$.
\end{proposition}
\begin{proof}

For each non-zero $P\in\mathbb{K} [t_1,\dots, t_N]$, let $M(P)$  and $m(P)$ respectively denote the largest and smallest moduli of the non-zero coefficients of $P$. So
\[
F=\cup_{m\in\N} F_m,
\]
where for each $m\in \N$
\[
F_m=\{ f=(f_1,\dots, f_N)\in X^N:\ \exists R\in \mathcal{P}_m: \ R(f_1,\dots, f_N)=0 \}
\]
and where $\mathcal{P}_m$ consists of those  $R$ in $\mathbb{K} [t_1,\dots, t_N]$ which simultaneously vanish at the origin, have degree in $[1,m]$ and $\frac{1}{m} \le m(R)\le  M(R) \le m$.
%
%\[
%\mathcal{P}_m=\{ R\in\mathbb{K} [t_1,\dots, t_N]:\ \mbox{deg}(R)\in [1,m] \mbox{ and } m(R), M(R)\in [\frac{1}{m}, m] \}.
%\]
Notice that each $F_m$ is closed by Remark~\ref{R:00}, thanks to the compactness of $\mathcal{P}_m$ and the continuity of the map
\[
\mathcal{P}_m\times X^N\to X, \ (R, f)\mapsto R(f).
\]
It suffices to show that $F_m$ has empty interior. By means of contradiction, suppose that $F_m$ contains a non-empty open subset $U$ of $X^N$.

\underline{Case 1}: $X$ has a unit $e$. By Remark~\ref{R:12}, we know that $X=\overline{A_1(Y)}$ for some subset $Y$ for which the set $\widetilde{Y}$ of finite products 
\[
y_1^{\alpha_1}\cdots y_n^{\alpha_n}
\]
with $n\in\N$, with $y_1,\dots, y_n\in Y$ pairwise distinct, and with $\alpha_1,\dots, \alpha_n\in\N_0$ not simultaneously zero, satisfies that $\widetilde{Y}\cup\{e\}$ is linearly independent.  Since $\times_{j=1}^N A_1(Y)$ is dense in $X^N$ there exist
 $g=(g_1,\dots, g_r)\in Y^r$ with $g_1,\dots , g_r$ pairwise distinct and $P=(P_1,\dots P_N)$ in $\mathbb{K} [t_1,\dots, t_r]^N$  so that
\[
P(g)=(P_1(g),\dots, P_N(g))\in U.
\]
Without loss of generality we may assume that $d_1:=\mbox{deg}(P_1)>0$ and
\begin{equation} \label{eq:deg}
d_{j}:=\mbox{deg}(P_j)>(m+1) d_{j-1}   \ \ (j=2,\dots, N).
\end{equation}
Since $P(g)\in F_m$ there exists $R\in\mathcal{P}_m$ so that
\begin{equation}\label{eq:1.4}
R(P_1(g),\dots, P_N(g))=0.
\end{equation}
Write \[ R=\sum_{\alpha\in A} c_\alpha t^\alpha,\]
where $A\subset \N_0^N$ is finite and non-empty and so that  $1\le |\alpha|\le m$ and $\frac{1}{m}\le |c_\alpha|\le m$ for each $\alpha\in A$. Now, let $A_0:=A$ and $\beta_N:=\mbox{max}\{ \alpha_N: \ \alpha\in A_0\}$, and inductively for $k=1,\dots, N-1$ let
\[
A_k:=\{ \alpha\in A_{k-1}: \ \alpha_{N-k+1}=\beta_{N-k+1}\}  \mbox{ and } \beta_{N-k}:=\mbox{max}\{ \alpha_{N-k}: \ \alpha\in A_k\}.
\]
Then $\beta=(\beta_1,\dots,\beta_N)\in A$, and for each $\alpha\in A\setminus\{\beta\}$ either $\alpha_N<\beta_N$ or there exists  $1\le k< N$ satisfying
\[
\alpha_k<\beta_k \mbox{ and } (\alpha_{k+1},\dots,\alpha_N)=(\beta_{k+1},\dots,\beta_N).
\]
In the former case, by $\eqref{eq:deg}$ we have
\begin{align*}
\mbox{deg}(P^\alpha)=\sum_{i=1}^N \alpha_i d_i &\le (m+1) d_{N-1} +\alpha_N d_N \\
&< \beta_N d_N \le \mbox{deg}(P^\beta).
\end{align*}
In the latter case, again by $\eqref{eq:deg}$ we have
\begin{align*}
\mbox{deg}(P^\alpha) = \sum_{i=1}^N \alpha_i d_i 
&\le (m+1) d_{k-1}+  \sum_{i=k}^N \alpha_i d_i            \\ 
&<  d_k + (\beta_k-1) d_k+    \sum_{i=k+1}^N \beta_i d_i                         \\  
&\le   \mbox{deg}(P^\beta).     %             
\end{align*}
That is, in either case the maximum degree of 
$
R(P)=\sum_{\alpha\in A} c_\alpha P^\alpha
$
occurs only at the term $P^\beta$, forcing by $\eqref{eq:1.4}$ the linear dependence of $\widetilde{Y}\cup\{ e\}$, a contradiction.

\underline{Case 2}: $X$ has no unit.   We follow the proof of Case 1, except that here we replace $A_1(Y)$ by $A(Y)=\mbox{span}(\widetilde{Y})$ where
only $\widetilde{Y}$ is linearly independent,
and the polynomials $P=(P_1,\dots, P_N)$ also satisfy $P_1(0)=\dots =P_N(0)=0$, forcing the non-zero polynomial $R(P)\in \mathbb{K}[t_1,\dots, t_r]$ to have no constant term, what together with $\eqref{eq:1.4}$ contradicts the linear independence of $\widetilde{Y}$. 
\end{proof}

\begin{lemma}\label{L:15}
Let $X$ be a separable $F$-algebra. Then the set of $f$ in $X^\N$  that induce a dense algebra on $X$ is residual in $X^\N$. 
\end{lemma}
\begin{proof}
We show that
$$
A=\{f=(f_n)_{n=1}^{\infty}\in X^{\N}: A(f) \,\, \mbox{is dense in} \,\, X\}
$$
is a dense $G_{\delta}$ subset of $X^{\N}$. 
Let $\{V_k\}_{k=1}^{\infty}$ be a countable base for the topology of $X$, and for each $k\in \N$ define
$$
A_k=\{ f \in X^{\N}: A(f)\cap V_k\neq \emptyset \}.
$$
It clearly holds that 
$
A=\bigcap_{k=1}^{\infty}A_k
$
so, it suffices to show that each set $A_k$ is open and dense in $X^{\N}$. To see that $A_k$ is open, let $f=(f_n)_{n=1}^{\infty}\in A_k$.   So there exists some $g\in A(f)\cap V_k$, say $g=P(f_1,\dots, f_r)$ for some $r\in\N$ and $P\in \mathbb{K}[t_1,\dots, t_r]$ with $P(0)=0$. By the continuity of the map $h\in X^r\mapsto P(h)\in X$, there exist $U_1,\dots ,U_r$ neighborhoods of $f_1,\dots, f_r$ respectively such that $P(h_1,\dots ,h_r) \in V_k$ for each $(h_1,\dots ,h_p)\in U_1\times \dots \times U_r$. So, the set $U_1\times \dots \times U_r\times X\times X\times \dots $ is an open neighborhood of $f$ contained in $A_k$.  To show the density of $A_k$, let $U_1\times \dots \times U_N\times X\times X\times \dots $ be a basic open subset of $X^{\N}$. Then for arbitrary $f_i \in U_i$  $(i=1,\dots ,N)$ and $g\in V_k$ we have
$$
(f_1,\dots ,f_N,g,g,\dots)\in [U_1\times \dots \times U_N\times X\times X\times \dots]\cap A_k.
$$ \end{proof}

\begin{remark}\label{R:16}
We note that Lemma~\ref{L:15} holds for arbitrary separable $F$-algebras. % Also, Lemma~\ref{L:15} does not hold if we replace $X^\N$ by $X^N$.
\end{remark}

We are ready to show Proposition~\ref{P:00}.
\begin{proof}[Proof of Proposition~\ref{P:00}]
Part (1) follows by  Proposition~\ref{P:07}, Lemma~\ref{L:15}, and Corollary~\ref{C:05}. To see (2), 
let $B$ denote the set of sequences in $X$ that are algebraically independent. So
\[
X^\N\setminus B \subseteq \cup_{N=2}^\infty F_N,
\]
where 
\[
F_N=\{ g\in X^\N:   \ (g_1,\dots, g_N) \mbox{ is algebraically dependent} \},
\]
and it suffices to show that each $F_N$ is meager in $X^\N$.  Now, by Proposition~\ref{P:07} the set
\[
S_N:= \{ h\in X^N:   \ (h_1,\dots, h_N) \mbox{ is algebraically dependent} \}.
\]
is meager in $X^N$, so $S_N\times X^\N$ is meager in $X^N\times X^\N$.
But the homeomorphism 
\[f\in X^\N \mapsto ((f_j)_{j=1}^N, (f_s)_{s>N})\in X^N\times X^\N\] identifies $F_N$ with the set $S_N\times X^\N$. So $B$ is residual in $X^\N$, and the conclusion follows by Lemma~\ref{L:15}.
\end{proof}

\section{Proof of Theorem~\ref{T:1}} \label{S:2}

The proof of Theorem~\ref{T:1} follows the approach used to show \cite[Theorem~3]{BCP2}  and is necessarily more technical as the latter only establishes singly generated algebras.  
%The first one, Lemma~\ref{L:2},  provides a general sufficient condition for a an operator $T$ on a Frech\'et algebra $X$ to support a residual set of elements $f$ in $X^N$ whose algebra $A(f)$ they generate is a hypercyclic algebra for $T$. In order to formulate it we provide in Definition~\ref{D:pivot} below the notion of a pivot of a finite set of multi-indexes. The second result, Lemma~\ref{L:7}, ensures the geometric assumptions on the entire function $\Phi$ in Theorem~\ref{T:1} suffice for the convolution operator $T=\Phi(D)$ on $X=H(\C)$ to satisfy the assumptions of Lemma~\ref{L:2}. 
We first provide a sufficient condition for the existence of finitely generated hypercyclic algebras that complements  \cite[Remark~8.28]{bayart_matheron2009dynamics} by Bayart and Matheron.
\begin{lemma} \label{L:2} Let $X$ be a separable commutative $F$-algebra, let $T\in L(X)$,
and let $N\ge 2$ be fixed.
Suppose that for each non-empty finite subset $A$ of $\N_0^N$ not containing the zero $N$-tuple there exists $\beta\in A$ satisfying:  

\begin{quote}
$(\ast)$ ``For each non-empty open subsets $U_1,\dots, U_N, V$ and $W$ of $X$ with $0\in W$ there exist
 $f\in U_1\times\dots\times U_N$ and $q\in\N$ so that $T^q(f^\beta)\in V$ and $
T^q(f^\alpha)\in W$  for each $\alpha\in A\setminus \{ \beta \}$''. 
%\begin{equation} \notag%\tag{$\ast$} \label{eq:2,3} 
%\begin{aligned}
%T^q(f^\beta)&\in V   \\
%T^q(f^\alpha)&\in W \ \ \mbox{ for each $\alpha\in A\setminus \{ \beta \}$''.} 
%\end{aligned}
%\end{equation}
\end{quote}
Then the set of $f=(f_1,\dots ,f_N)$ in $X^N$ that generate a hypercyclic algebra for $T$ is residual in $X^N$.
\end{lemma}
\begin{proof}  Let $\{ V_k \}_{k\in\N }$ be a countable base for the topology of $X$, and let  $\{ W_k \}_{k\in\N }$ be a countable balanced local base. Also, let $\Lambda_N$ denote the set of all finite and non-empty subsets of $\N_0^N$ that do not contain the zero $N$-tuple. For each $A\in \Lambda_N$, let $\mathcal{R}_A$ denote the (non-empty, by our assumption) 
subset of elements $\beta$ of $A$ that satisfy $(\ast)$. 
Now, for each $(k,\ell )\in \N\times\N$ and each $A\in \Lambda_N$ and $\beta\in \mathcal{R}_A$ consider the set $\mathcal{A}(k,\ell, A,\beta)$ of $f\in X^N$ for which there exists some positive integer $q$ satisfying
\[
\begin{aligned}
T^q(f^\beta)&\in V_k \\
T^q(f^\alpha)&\in W_\ell \ \ \ (\alpha\in A\setminus\{ \beta \}).
\end{aligned}
\]
Each $\mathcal{A}(k,\ell ,A,\beta)$ is clearly open, and it is dense by our assumption $(\ast)$. Baire's Category theorem gives that the countable intersection
\[
G:=\underset{k,\ell \in\N}{\cap} \ \underset{A\in \Lambda_N}{\cap} \ \underset{\beta\in \mathcal{R}_A}{\cap} \ \mathcal{A}(k,\ell ,A,\beta)
\]
is residual in $X^N$. But each $f\in G$ induces a hypercyclic algebra for $T$. To see this,  let $0\ne P\in \mathbb{K}[z_1,\dots ,z_N]$ with $P(0)=0$. We want to show that $P(f)$ is hypercyclic for $T$. Without loss of generality, 
\[
P(z)=\sum_{\alpha\in  A} c_\alpha z^\alpha
\]
for some $A\in \Lambda_N$ so that $c_\alpha\ne 0$ for each $\alpha\in A$. Let $V$ be an arbitrary non-empty open subset of $X$. 
Let $\beta\in \mathcal{R}_A$, and let $(k,\ell)\in \N\times\N$ so that

\begin{equation} \label{eq:lego}
c_\beta V_k + \sum_{\alpha\in A\setminus\{ \beta \} } c_\alpha W_\ell \subset V.
\end{equation}
Then since $f$ belongs to $\mathcal{A}(k,\ell ,A,\beta)$ there exists $q\in\mathbb{N}$ so that $T^qf\in V_k$ and $T^qf^\alpha \in W_\ell$ for each $\alpha\in A\setminus \{\beta \}$. Hence by $\eqref{eq:lego}$  
\[
T^qP(f)= c_\beta T^qf^\beta    +    \sum_{\alpha\in A\setminus\{ \beta \} } c_\alpha T^qf^\alpha \in V.                                                                                           
\]
\end{proof}
We next establish in Lemma~\ref{L:7-} below some consequences of the geometric assumption in Theorem~\ref{T:1}.
Recall that for  a planar smooth curve $\mathcal{C}$ with parametrization $\gamma : [0,1]\to \C$, $\gamma (t)=x(t)+i y(t)$, its signed curvature at a point $P=\gamma (t_0)\in \mathcal{C}$ is given by
\[
\kappa (P):= \frac{  x'(t_0) y''(t_0) - y'(t_0) x''(t_0) }{ |\gamma'(t_0)|^3 }.
\]
and its unsigned curvature at $P$ is given by $|\kappa (P)|$.
It is well-known that $|\kappa (P)|$ does not depend on the parametrization chosen, and that the signed curvature $\kappa (P)$ depends only on the choice of orientation seleted for $\mathcal{C}$.  Also, in the particular case when $\mathcal{C}$ is given by the graph of a function $y=f(x)$, $a\le x\le b$, (and oriented from left to right), its signed curvature at a point $P=(x_0, f(x_0))$ is given by
\[
\kappa (P)= \frac{ y''(x_0)}{(1+(y'(x_0))^2 )^{\frac{3}{2}}}.
\]
In particular, $\kappa < 0$ on $\mathcal{C}$ if and only if $y=f(x)$ is concave down (i.e.,  $ (1-s) f(a_1) + s f(b_1)<  f((1-s)a_1 + s b_1) $     for  any $s\in (0, 1)$ and any subinterval $[a_1, b_1]$ of $[a,b]$).

\begin{lemma} \label{L:7-}
Let $\Phi\in H(\C )$ be %an entire function
 of finite exponential  type with $|\Phi (0)|<1$ and so that the level set $\{ z: \ |\Phi(z)|=1 \}$ contains a non-trivial, strictly convex compact arc $\Gamma_1$ satisfying
\begin{equation} \label{eq:L7-.1}
\mbox{conv}(\Gamma_1\cup \{ 0 \} ) \setminus  \Gamma_1 \subseteq \Phi^{-1} (\mathbb{D}).
\end{equation}
Then for any integers $d, M$ with $1\le M\le d$ 
there exist a non-trivial compact segment $\Lambda \subset \Phi^{-1}(\mathbb{D})\setminus\{ 0\}$      and  $r_0>1$ so that for each $1<r\le r_0$ there exists a non-trivial strictly convex compact arc $\Gamma_r\subset \Phi^{-1}(r\partial \mathbb{D})$ with
\[
\mbox{conv}(\Gamma_r\cup\{ 0 \})\setminus \Gamma_r \subset \Phi^{-1}(r\mathbb{D})
\]
and satisfying
\begin{equation} 
\begin{aligned}
&(i)\ \  \ \hspace{.3in}  \sum_{i=1}^{d} \Lambda \subset \Phi^{-1}(\mathbb{D}), \\
&(ii)\ \  \hspace{.3in}     \mbox{conv}(\Gamma_r \cup \{ 0 \})      +   \sum_{s=1}^i \Lambda  \subset \Phi^{-1}(\mathbb{D}) \ \ \ (1\le i < d), \mbox{ and }\\
&(iii) \ \  \hspace{.3in}  \sum_{s=1}^i   \frac{1}{M} \Gamma_r \subset \Phi^{-1}(\mathbb{D})   \ \ \ \ (1\le i < M).
%&(iv) \ \ \hspace{.3in}  \mbox{conv}(\Gamma_r\cup\{ 0 \})\setminus \Gamma_r \subseteq \Phi^{-1}(r\mathbb{D}).
\end{aligned}
\end{equation}

\end{lemma}
\begin{proof} For each $0\ne z\in\mathbb{C}$ we denote by $\mbox{arg}(z)$ the argument of $z$ that belongs to $[0, 2\pi)$.
Since $\Gamma_1$ is strictly convex, replacing it by a subarc  and $\Phi(z)$ by $\tilde{\Phi}(z)=\Phi(az)$ for some $a\in\partial\mathbb{D}$ if necessary (see \cite[Remark~6]{BCP2} ) we may assume that $\Gamma_1$ is simple and with endpoints $z_1$, $z_2$ satisfying
\begin{equation} \label{eq:endpoints}
\begin{aligned}
0<\mbox{arg}(z_1) & < \mbox{arg}(z_2) < \pi \\
\mbox{Re}(z_2) &< \mbox{Re}(z_1) 
\end{aligned}
\end{equation}
and so that $0\notin \mbox{conv}(\Gamma_1)$.
Thus (cf \cite[Proposition~7]{BCP2}) 
\[
\Omega:=\mbox{conv}(\Gamma_1\cup \{ 0 \})\setminus (\Gamma_1 \cup \{ 0\} ) = \{ tz:\ (t,z)\in (0,1)\times \mbox{conv}(\Gamma_1) \} %\subset S(0, z_1, z_2),
\]
is contained in the sector $\{ 0\ne w\in\mathbb{C}:  \ \mbox{arg}(z_1)\le \mbox{arg}(w)\le \mbox{arg}(z_2) \}$ and
 $\partial\Omega=[0,z_1)\cup\Gamma_1\cup(0,z_2)$, and we may assume that $\Gamma_1$ is the graph of a concave down function $f:[\mbox{Re}(z_2), \mbox{Re}(z_1)]\to (0,\infty)$ (replacing $z_j$ by $z_j'=\mbox{Re}(z_j)+if(\mbox{Re}(z_j))\in \Gamma_1$, $j=1,2$ if necessary).
Now, let $\epsilon >0$ so that $D(0,\epsilon)\subset \Phi^{-1}(\mathbb{D})$ and pick  $0\ne z\in D(0,\epsilon)\cap \mbox{conv}(\Gamma_1\cup\{ 0\})$ close enough to zero so that its additive inverse $\tilde{z}=-z$ satisfies, replacing $\Gamma_1$ by a subarc with endpoints satisfying $\eqref{eq:endpoints}$ as well as $z_1$ and $z_2$ by the endpoints of this subarc and restricting the function $f$ to the corresponding subinterval $[\mbox{Re}(z_2),\mbox{Re}(z_1)]$, that
\[
\mbox{conv}(\Gamma_1\cup \{ \tilde{z} \})\setminus\Gamma_1 \subset \Phi^{-1}(\mathbb{D})
\]
and $0\in\mbox{int}(\mbox{conv}(\Gamma_1\cup\{ \tilde{z} \}))$. In particular, letting  $0<\epsilon_1$ so that $D(0,\epsilon_1)\subset \mbox{conv}(\Gamma_1\cup \{ \tilde{z} \})\setminus\Gamma_1$,
for each $z^*\in D(0,\epsilon_1)$ we have
\begin{equation} \label{eq:2and3/4}
\mbox{conv}(\Gamma_1\cup\{ z^*\})\setminus \Gamma_1 \subset \Phi^{-1}(\mathbb{D}).
\end{equation}
Now, pick $z_0\in \Gamma_1\setminus \{ z_1, z_2 \}$ with $\Phi' (z_0)\ne 0$, and let $w_0:=\Phi(z_0)=e^{i\theta_0}$, where $\theta_0\in [0, 2\pi)$. Choose $\rho>0$ small enough so that the only solution to
\[
\Phi(z)=w_0
\]
in $D(z_0, \rho)$ is at $z=z_0$, and so that $D(z_0, \rho )\cap ([0, z_1] \cup [0, z_2])=\emptyset$.  Next, pick
\[
0<s<\mbox{min}\{ |\Phi(z)-w_0|: \ |z-z_0|=\rho \}
\]
and let $0<\delta < \mbox{min}\{ 1, s\}$ so that the polar rectangle
\[
R_\delta := \{ z=r e^{i\theta }: \ (r,\theta)\in [1-\delta, 1+\delta]\times [\theta_0-\delta, \theta_0+\delta] \}
\]
is contained in $D(w_0, s)$. Then 
\[
g:R_\delta \to D(z_0, \rho), \ g(w)= \frac{1}{2\pi i} \underset{|z-z_0|=\rho}{\int} \frac{z \Phi'(z)}{\Phi (z)-w } dz
\]
defines a univalent holomorphic function satisfying that
\begin{equation} \label{eq:2}
\Phi \circ g = \mbox{identity on $R_\delta$,}
\end{equation}
see e.\ g.\ \cite[p.\ 283]{gamelin}.
So $W:= g(R_\delta)$ is a connected compact neighborhood of $z_0$, and $\Phi$ maps $W$ biholomorphically onto $R_\delta$. Hence for each $1-\delta \le r \le 1+\delta $
\[
\eta_r := g(R_\delta \cap r\partial \D)
\]
is a smooth arc contained in $W\cap \Phi^{-1}(r\partial\D)$. In particular, $\eta_1=W\cap\Gamma_1$ is a strictly convex subarc of $\Gamma_1$. 
Next, notice that since 
\[
W\cap \Omega \ \ \mbox{ and } \ \ W\cap \mbox{Ext}(\Omega ) 
\]
are the two connected components of $g(R_\delta\setminus \partial\D )= W\setminus \eta_1$
and $\Omega \subseteq \Phi^{-1}(\D)$, by $\eqref{eq:2}$  the homeomorphism $g:R_\delta\setminus \partial\D \to W\setminus \eta_1   $ must satisfy
\[
\begin{aligned}
g(R_\delta \cap \mbox{Ext}(\D)) &= W\cap \mbox{Ext}(\Omega) \\
g(R_\delta \cap \D)&= W\cap \Omega.
\end{aligned}
\]
Hence 
\[
W\cap \overline{\mbox{Ext}(\Omega)} = \underset{1\le r \le 1+\delta}{\cup} \eta_r
\]
and $g$ induces a smooth homotopy among the curves $\{ \eta_r \}_{1\le r\le 1+\delta}$. Namely,
each $\eta_r$ $(1\le r\le 1-\delta)$ has the Cartesian parametrization 
\[
\eta_r: \ \begin{cases}   X(r,t) \\
Y(r,t)
\end{cases}  
\ \ \ \theta_0-\delta \le t \le \theta_0 +\delta,
\]
where $X, Y:[1-\delta, 1+\delta]\times [\theta_0-\delta, \theta_0+\delta]\to \R$ are given by
\[
\begin{aligned}
X(r,t)&:= \mbox{Re}(g)(r e^{it}) \\
Y(r,t)&:= \mbox{Im}(g)(r e^{it}).
\end{aligned}
\]
Now, for each $P=g(re^{i\theta})$ in $W$ the (signed) curvature $\kappa^{\eta_r}(P)$ of $\eta_r$ at $P$  is given by 
\[
\kappa^{\eta_r}(P) = \frac{        \frac{   \partial X }{\partial t      } (r,\theta)  \frac{   \partial^2 Y }{\partial^2 t      } (r,\theta) - \frac{   \partial Y }{\partial t      } (r,\theta)  \frac{   \partial^2 X }{\partial^2 t      } (r,\theta)  }{     \left( ( \frac{\partial X}{\partial t}(r,\theta))^2 +    ( \frac{\partial Y}{\partial t}(r, \theta))^2 \right)^{\frac{3}{2}} },
\]
Hence the map $K : W\to \R$, $K (g(re^{it})) := \kappa^{\eta_r}(P)$, is continuous.  Now, since $\eta_1$ is strictly convex there exists some $P=g(e^{i\theta_1})$ in $\eta_1$ for which each of $\kappa^{\eta_1}(P)$, $\frac{\partial X}{\partial t}(1,\theta_1)$ is non-zero. Hence by the continuity of $K$ and of $\frac{\partial X}{\partial t}$ we may find some $0<\delta'<\delta$ so that the polar rectangle 
\[
R_{\delta'} := \{ z=r e^{i\theta }: \ (r,\theta)\in [1-\delta', 1+\delta']\times [\theta_1-\delta', \theta_1+\delta'] \}
\]
is contained in the interior of $R_\delta$ and so that $K$  and $\frac{\partial X}{\partial t}$  are bounded away from zero on $g(R_{\delta'})$ and on $R_{\delta'}$, respectively. 

In particular, either  $\frac{\partial X}{\partial t}>0$ \ or  $\frac{\partial X}{\partial t}<0$ on $R_{\delta'}$, and 
either $K>0$ or $K<0$ on $g(R_{\delta'})$. So each $\eta_r\cap g(R_{\delta'})$ $(1\le r < 1+\delta')$ is the graph of a smooth function 
\[
f_r:(a_r, b_r)\to (0, \infty),
\]
with \[
(a_r, b_r)=\begin{cases} (X(r, \theta_1-\delta'), X(r, \theta_1+\delta')) &\mbox{ if $\frac{\partial X}{\partial t}>0$ on $R_{\delta'}$} \\
(X(r, \theta_1+\delta'), X(r, \theta_1-\delta')) &\mbox{ if $\frac{\partial X}{\partial t}<0$ on $R_{\delta'}$.} 
\end{cases}
\]

Notice that $g(re^{it})\underset{r\to 1}{\to} g(e^{it})$ uniformly on $t\in [\theta_1-\delta, \theta_1+\delta]$, so \[(a_r, b_r)\underset{r\to1}{\to}(a_1, b_1)\] and fixing a non-trivial compact subinterval $[a,b]$ of $(a_1, b_1)$  there exists $1<r_0$ so that
\[
[a,b]\subset \cap_{1\le r\le r_0} (a_r, b_r).
\]
Thus for each $1\le r\le r_0$ we have that %$f_r$ is defined on $[a,b]$ and
\[
\eta_r'=\{ (x, f_r(x)); \ x\in [a,b] \}
\]
is a subarc of $\eta_r$. Moreover, since $f_1=f$ on $[a, b]$ it must be a concave down function, and so we must have $K<0$ on $g(R_{\delta'})$. Hence each  $f_r$ with $1\le r\le 1+\delta'$ is also concave down. and for any $1<r<r_0$ close enough to $1$ the arc $\Gamma:=\eta_r'$
satisfies
%\[
\begin{equation} \label{eq:3.7a}
\mbox{conv}(\Gamma\cup\{ 0\})\setminus (\Gamma\cup\{ 0\}) \subset \Phi^{-1}(r\D )
\end{equation}
and
\begin{equation}\label{eq:3.7b}
\sum_{s=1}^i \frac{1}{M}\Gamma \subset \Omega \ \ \mbox{for $1\le i< M$.}
\end{equation}
Further reducing $\Gamma_1$ if necessary, we may assume that $\Gamma_1=\eta_1'$ (so $z_2=a+if(a)$ and $z_1=b+if(b)$). 
Now, let $z_3:=\frac{a+b}{2}+i f(\frac{a+b}{2})\in \Gamma_1$ and let $z_3^*$ be a point in 
\[
D(0,\epsilon_1)\cap \{ tz_3: \ t<0\}.
\]
By $\eqref{eq:2and3/4}$, we know that
%\begin{equation} \label{eq:13and1/2}
\[
\mbox{conv}(\Gamma_1\cup\{ z_3^*\})\setminus \Gamma_1 \subset \Phi^{-1}(\mathbb{D}).
\]
%\end{equation}
Also, let $\Lambda \subset (z_3^*, 0)$ be a compact segment small enough and close enough to $0$ so that
\[
\sum_{s=1}^{d} \Lambda \subset (z_3^*, 0).
\]
Let $\lambda_0\in\Lambda$ so that $|\lambda_0|=\mbox{dist}(0,\Lambda)$, and let $\lambda_{-1}\in (z_3^*,0)$ so that 
\[
\sum_{s=1}^i \Lambda \subset [\lambda_{-1}, \lambda_0] \ \ \ (1\le i<d).
\]
For $j=1,2$, let $w_j$ denote the point at which the arc  $\Gamma_1+\lambda_0+\lambda_{-1}$ meets $[z_3^*, z_j]$. Also, let $\gamma_{-1}$ denote the subarc of $\Gamma_1+\lambda_0$ with endpoints 
\[
A_{j,-1}:= w_j-\lambda_{-1} \ \ (j=1,2)
\]
and let $\gamma_0:=\gamma_{-1}-2\lambda_0$ denote the subarc of $\Gamma_1-\lambda_0$ with endpoints
\[
A_{j,0}:=A_{j-1}-2\lambda_0 \ \ (j=1,2).
\]
So if $V$ denotes the bounded open region with boundary
\[
\partial V=\gamma_{-1} \cup [A_{1,-1}, A_{1,0}]\cup \gamma_{0} \cup [A_{2,0}, A_{2,-1}]
\]
then $z_3\in V$ and
\[
 V+\sum_{s=1}^i \Lambda \subset \mbox{conv}(\Gamma_1\cup\{ z_3^*\})\setminus \Gamma_1 \subset \Phi^{-1}(\mathbb{D}) \ \  \ (1\le i< d).
\]
Let $Q$ be a non-trivial closed rectangle centered at $z_3$ with base parallel to the $x$-axis and
\[
Q\subset V\cap ( (a,b)\times (0,\infty)).
\]
Next, choose $1< r\le r_0$ close enough to $1$ so that $\eqref{eq:3.7a}$ and $\eqref{eq:3.7b}$ hold and so that the subarc $\eta_r'$ of $\eta_r$ intersects $Q$. Then 
\[
\Gamma_r := \{ (x, f_r(x)): \ x\in[a,b] \} \cap Q  \}
\]
is a non-trivial strictly convex compact arc in $\Phi^{-1}(r\partial\mathbb{D})$ satisfying
\[
\mbox{conv}(\Gamma_r \cup \{ 0 \}) \setminus \Gamma_r \subset \Phi^{-1} ( r \mathbb{D} ),
\]
as well as
\[
\sum_{s=1}^i \frac{1}{M}\Gamma_r \subset \Omega\subset\Phi^{-1}(\mathbb{D})     \ \ \ (1\le i< M)
\]
and 
\[
\mbox{conv}(\Gamma_r\cup \{0 \}) + \sum_{s=1}^i \Lambda \subset \Phi^{-1}(\mathbb{D}) \ \ \ (1\le i< d).
\]
\end{proof}

%We next establish the second result we use to show Theorem~\ref{T:1}.
Finally, we note the following fact about finite sets of $N$-dimensional multi-indices. We thank an anonymous referee for providing the proof.

\begin{lemma} \label{L:1}
Let  $A$ be a finite non-empty subset of $\mathbb{N}_0^N$, where $N\ge 2$.
Then there exist positive scalars $k_i$ 
 $(i=1,\dots, N)$   so that the functional
\[
(x_i)_{i=1 }^N
\mapsto \sum_{i=1}^N k_i x_i
\]
is injective on $A$. 
\end{lemma}

\begin{proof}\  
%Any functional on a finite subset of $\mathbb{R}^N$ attains a minimum, and the conclusion is trivial when $A_{i_A}$ is a singleton. So for the existence of a strict minimum 
The conclusion is trivial if $A$ is a singleton, so we may assume
the finite set \[ B:=\{ \alpha-\alpha' \in \mathbb{R}^N:  \alpha, \alpha'\in A,  \alpha\ne \alpha' \}\] is non-empty.
It suffices to find  $k\in (0,\infty)^N$ for which the functional
\[
(x_i)_{i=1 }^N
\mapsto \sum_{i=1}^N k_i x_i
\]
is zero-free on $B$. 
 Now, for each $\gamma \in B$  the functional $\psi_\gamma: (0,\infty)^N\to \mathbb{R}$, $\psi_\gamma (k)= \sum_{i=1}^N k_i \gamma_i$ is non-trivial  and thus $\psi_\gamma^{-1}(\{ 0\})$ is of first category in $(0,\infty )^N$. So the finite union
\[
\cup_{\gamma\in B} \ \psi_\gamma^{-1}(\{ 0\})
\]
 is of first category in $(0,\infty)^N$. But for any  $k$ in  $(0,\infty)^N\setminus \cup_{\gamma\in B} \psi_\gamma^{-1}(\{ 0\})$, we have that  $\psi_\gamma (k)\ne 0$ for every $\gamma \in B$.
\end{proof}

We are ready now to show the main result.
\begin{proof}[Proof of Theorem~\ref{T:1}:]  
Let $N\ge 2$ and let $A$ be a non-empty finite subset of $\mathbb{N}_0^N$ not containing the zero $N$-tuple. By Lemma~\ref{L:1} there exists $k\in (0,\infty)^N$ so that the real functional
\[
(x_i)_{i=1 }^N
\mapsto \sum_{i=1}^N k_i x_i
\]
is injective on $A$. Let $\beta\in A$ be the point at which the  functional attains its strict minimum over the set $\{ \alpha\in A: \ \alpha_{i_A}=M_A \}$, 
 where
$M_A:=\mbox{max}\{ |\alpha|_\infty: \alpha\in A\}$ and where
\[
i_A:=\mbox{min} \{ 1\le j \le N: \ \exists \alpha=(\alpha_1,\dots, \alpha_N)\in A \mbox{ with } \alpha_j=M_A \}.
\] 
So we have
 \begin{equation}\label{eq:k}
\sum_{1\le i\le N,  \  i\ne i_A}  k_i (\beta_i-\alpha_i) < 0 \ \  \ \mbox{ for each $\alpha\in A\setminus\{\beta \}$ with $\alpha_{i_A}=M_A$.}
\end{equation}
It suffices to show the following claim: 

\begin{claim} \label{c:2}
 Let  $U_1,\dots, U_N, W, V$ be non-empty open subsets of  $H(\C )$ with $0\in W$. 
Then there exist $f\in U_1\times\cdots\times U_N$ and $q\in\N$ so that
\begin{equation}
\begin{aligned}
\Phi(D)^q(f^\beta)&\in V   \\
\Phi(D)^q(f^\alpha)&\in W \ \ \mbox{ for each $\alpha\in A\setminus \{ \beta \}$.}
\end{aligned}
\end{equation}
\end{claim}

Indeed, suppose Claim~\ref{c:2}  holds. By Lemma~\ref{L:2}, the operator $T=\Phi(D)$ acting on the separable $F$-algebra $X=H(\mathbb{C})$ satisfies that
the set $H_N$ consisting of  those $f=(f_1,\dots ,f_N)$ in  $H(\mathbb{C})^N$ that generate a hypercyclic algebra for $\Phi(D)$ is residual in $H(\mathbb{C})^N$. So part $(a)$ of Theorem~\ref{T:1} follows by Proposition~\ref{P:00}, and  for each $N\ge 2$ we have that
\[
G_N:=\{ g=(g_j)\in H(\mathbb{C})^\N:\ (g_1,\dots ,g_N)\in H_N \} 
\]
is residual in $H(\mathbb{C})^\N$, and so is the countable intersection 
\[
G:=\cap_{N\ge 2} G_N.
\]
Now, given $f=(f_j)\in G$ the algebra it generates may be written as
\[
A(f)=\cup_{N\ge 2} \{ P(f_1,\dots, f_N):\ 0\ne P\in\mathbb{C}[z_1,\dots, z_N]  \mbox{ with } P(0)=0 \}
\]
But given any $N\ge 2$ and $0\ne P\in\mathbb{C}[z_1,\dots, z_N]$ with $P(0)=0$, we know since $f\in G_N$ that $P(f_1,\dots,f_N)$ is hypercyclic for $\Phi(D)$. That is, $A(f)$ is a hypercyclic algebra and part $(b)$ of Theorem~\ref{T:1} follows by Proposition~\ref{P:00} and Lemma~\ref{L:15}.

Now, to show Claim~\ref{c:2} notice that by  Lemma~\ref{L:7-} there exist $r>1$, a non-trivial strictly convex compact arc $\Gamma_r\subset \Phi^{-1}(r\partial\mathbb{D})$, 
and a non-trivial compact segment $\Lambda \subset \Phi^{-1}(\mathbb{D})\setminus\{ 0\}$ so that
\begin{equation}  \label{eq:sufficient}
\begin{aligned}
&(i)\ \  \ \hspace{.3in}  \sum_{i=1}^{d_A} \Lambda \subset \Phi^{-1}(\mathbb{D}), \\
&(ii)\ \  \hspace{.3in}     \mbox{conv}(\Gamma_r \cup \{ 0 \})      +   \sum_{s=1}^i \Lambda  \subset \Phi^{-1}(\mathbb{D}) \ \ \ (1\le i < d_A), \\
&(iii) \ \  \hspace{.3in}  \sum_{s=1}^i   \frac{1}{M_A} \Gamma_r \subset \Phi^{-1}(\mathbb{D})   \ \ \ \ (1\le i < M_A), \mbox{ and } \\
&(iv) \ \ \hspace{.3in}  \mbox{conv}(\Gamma_r\cup\{ 0 \})\setminus \Gamma_r \subseteq \Phi^{-1}(r\mathbb{D}),
\end{aligned}
\end{equation}
where $d_A:=\mbox{max}\{ |\alpha| : \alpha\in A \}$.
Since $\Gamma_r$ and $\Lambda$ have accumulation points in $\C$, there exist $B\in V$ and $L_i\in U_i$ $(i=1,\dots, N)$ of the form
\[
\begin{aligned}
B=B(z)&=\sum_{j=1}^p b_j e^{\gamma_j z} \\
L_i=L_i(z)&=\sum_{j=1}^p a_{i,j} e^{\lambda_{i,j} z}    \ \ \ (i=1,\dots,N)
\end{aligned}
\]
with $p>N$ and scalars $b_j, a_{i,j}\in\mathbb{C}$, $\gamma_j\in\Gamma_r$, and  $\lambda_{i,j}\in\Lambda$  $(1\le i\le N, \ %\mbox{ and } 
1\le j\le p)$.
Now, let $I_N:=\{ 1,\dots , N \} \setminus\{ i_A \}$, and for each $n\in\N$ set
\[
R_n:=\sum_{j=1}^p c_j \ e^{\frac{\gamma_j}{M_A} z},
\]
where $c_j=c_j(n)$ is a solution to
\begin{equation} \label{eq:c_j}
z^{M_A} (\Phi(\gamma_j))^n = b_j \ n^{\sum_{s\in I_N} k_s\beta_s}     \ \ \ (j=1,\dots p).
\end{equation}
Notice that $c_j=c_j(n)\underset{n\to\infty}{\to} 0$  $(j=1,\dots, p)$ and thus
\[
  L_{i_A}+R_n\in U_{i_A}     \  \mbox{ and }\    L_i+\frac{1}{\ n^{k_i}}\in U_i \ (i\in I_N)
\]
whenever $n$ is large enough. So letting
\[
f_i:=\begin{cases} 
\ L_{i_A}+R_n \ \ &\mbox{ if } i=i_A \\
\  \  L_i \, +\frac{1}{\ n^{k_i}}  \ &\mbox{ if } i\in I_N,
\end{cases}
\]
we have $f\in U_1\times\dots\times U_N$ whenever $n$ is large. Now, for each $\alpha=(\alpha_1,\dots ,\alpha_N) \in A$ and $n\in\N$ we have
\[
\begin{aligned}
\Phi(D)^n (f^\alpha) &= \Phi(D)^n (\prod_{i=1}^N f_i^{\alpha_i}) \\
&= \sum_{(u, v, \ell)\in \mathcal{I}_\alpha}  X_\alpha(u, v, \ell, n) \ e^{ (\lambda\cdot u + \frac{1}{M_A} \gamma\cdot v)z},
\end{aligned}
\]
where 
$
\mathcal{I}_\alpha$ consists of those multi-indexes 
\[
(u, v, \ell)=( (u_1,\dots, u_N), v, \ell) \in (  (\N_0^p )^N \times \N_0^p \times (\underset{i\in I_N}{\times} \N_0))\] for which
$
 |u_i|\, +\, \ell_i\ =\alpha_i  \ (i\in I_N)  \mbox{ and } |u_{i_A}|+|v|=\alpha_{i_A},
 $
 and where each $X_\alpha(u,v,\ell, n)$ is given by
\begin{equation} \label{eq:3.12}
%\begin{aligned}
X_\alpha(u,v,\ell,n) = {\alpha_{i_A}\choose{ u_{i_A}  v}  } \prod_{i\in I_N} {\alpha_i\choose{ u_i\, \ell_i }}   \ a^u c^v  \ \frac{ (\Phi (\lambda\cdot u + \frac{1}{M_A}\gamma\cdot v))^n}{ n^{\sum_{s\in I_N k_s \ell_s}}}, 
%\\ 
%&=  \frac{ {\alpha_{i_A}\choose{ u_{i_A} v}  } \prod_{i\in I_N} {\alpha_i\, \choose{ u_i\, \ell_i }}   \ a^u  b^{\frac{1}{M_A}v} }{ n^{\sum_{s\in I_N} (k_s \ell_s -  \frac{|v|}{M_A}     k_s\beta_s    )   }}   \left( \frac{ \Phi (\lambda\cdot u+\frac{1}{M_A} \gamma\cdot v) }{ \prod_{j=1}^p    \Phi(\gamma_j)^{\frac{v_j}{M_A}} } \right)^n,
%\end{aligned}
\end{equation}
where
\[
\begin{aligned}
\lambda\cdot u+\frac{1}{M_A} \gamma\cdot v&= \sum_{i=1}^N\sum_{j=1}^p \lambda_{i,j} u_{i,j}  +\frac{1}{M_A}  \sum_{j=1}^p \gamma_j v_j,\\
a^u&=\prod_{ 1\le i \le N, \  1\le j \le p }
%{(i,j)\in I_N\times I_p}.   
 a_{i,j}^{u_{i,j}}, \mbox{ and }\\ 
 %b^{\frac{1}{M_A}v}=\prod_{j=1}^p b_j^{\frac{v_j}{M_A}}, \ \
c^v&= \prod_{j=1}^p c_j^{v_j}.
\end{aligned}
\]

So it suffices to show that for each $\alpha\in A\setminus\{ \beta \}$
\begin{equation}  \label{eq:19}
X_\alpha (u, v, \ell, n)\underset{n\to\infty}{\to} 0 \ \ \ \ \ ((u,v,\ell)\in \mathcal{I}_\alpha)
\end{equation}
and that for each $(u,v,\ell)\in\mathcal{I}_\beta$
\begin{equation}  \label{eq:20}
\lim_{n\to\infty} X_\beta(u,v,\ell, n)=\begin{cases}
b_j     &%\mbox{ if $u=0$ and $v=M_A e_j$ for some $1\le j\le p$}\\
\mbox{ if $(u, v)\in \{ (0, M_A e_j): j=1,\dots, p\}$ }\\
0 &\mbox{ otherwise,}
\end{cases}
\end{equation}
where $\{ e_1,\dots ,e_p\}$ is the standard basis of $\mathbb{C}^p$.
Now, let $\alpha\in A$ and let $(u,v,\ell)\in\mathcal{I}_\alpha$ be given. Notice that by $\eqref{eq:c_j}$ and $\eqref{eq:3.12}$ we have
\[
\left| X_\alpha(u,v,\ell, n)\right|  \le (\mbox{Constant}) \ n^{\sum_{s\in I_N} (\frac{|v|}{M_A} k_s\beta_s - k_s\ell_s)} \,
 \left| \frac{ \Phi (\lambda\cdot u+\frac{1}{M_A} \gamma\cdot v) }{ \prod_{j=1}^p    \Phi(\gamma_j)^{\frac{v_j}{M_A}} } \right|^n.
\]
If $1\le |u|< d_A$ we have %then as each $|\Phi (\gamma_j)|=r>1$ we have
\[
|X_\alpha (u,v,\ell, n)| \le (\mbox{Constant} )\ n^{|v|\sum_{s\in I_N} k_s} \left| \Phi (\lambda\cdot u+\frac{1}{M_A} \gamma\cdot v)\right|^n \underset{n\to\infty}{\to} 0
\]
since  $|\Phi (\gamma_j)|=r>1$ for each $j=1,\dots, p$ and since $
\lambda\cdot u+\frac{1}{M_A} \gamma\cdot v \in \Phi^{-1}(\mathbb{D})$
 by $\eqref{eq:sufficient}$(ii).
 On the other hand, if $|u|=d_A$  (So $|v|=0=\ell_s$ for each $s\in I_N$), then by  $\eqref{eq:sufficient}$(i) we have
\[
|X_\alpha (u,v,\ell, n)| \le (\mbox{Constant} ). |\Phi (\lambda\cdot u)|^n \underset{n\to\infty}{\to} 0.
\]
 Finally, if $|u|=0$  (so $\ell_i=\alpha_i$ for each $i\in I_N$, and $|v|=\alpha_{i_A}$) we have two cases:\vspace{.1in}

\underline{Case 1:}  $\alpha_{i_A}< M_A$.\vspace{.05in}
In this case we again have $\lambda\cdot u+\frac{1}{M_A} \gamma\cdot v\in \Phi^{-1}(\mathbb{D})$ so by $\eqref{eq:sufficient}(iii)$ %and to $|\Phi(0)|<1$ 
\[
X_\alpha (u,v,\ell, n)\underset{n\to\infty}{\to} 0.
\]

\underline{Case 2:}  $\alpha_{i_A}=M_A$.\vspace{.05in}
Here we have two sub-cases: $|v|_\infty<|v|=M_A$, or else $v=M_A e_j$ for some $1\le j\le p$. If $|v|_\infty<M_A$ then since $\Gamma_r$ is strictly convex by $\eqref{eq:sufficient}$(iv) we have
\[
\lambda\cdot u+\frac{1}{M_A} \gamma\cdot v =\frac{1}{M_A}\gamma\cdot v\in\mbox{conv}(\Gamma_r)\setminus\Gamma_r\subset \Phi^{-1}(r\mathbb{D})
\]
and thus
\[
|X_\alpha(u,v,\ell, n)| \le (\mbox{Constant}) \ n^{\sum_{s\in I_N} k_s (\beta_s-\alpha_s)} \left( \frac{|\Phi (\lambda\cdot u+\frac{1}{M_A} \gamma\cdot v)|}{r}\right)^n \underset{n\to\infty}{\to} 0.
\]
Else, if $v=M_A e_j$ for some $1\le j\le p$, then for $\alpha=\beta$ we have 
\[
X_\beta(u, v,\ell, n)= b_j \ \ \ (n\in\N)
\] by $\eqref{eq:c_j}$ and $\eqref{eq:3.12}$, while for $\alpha\in A\setminus\{\beta\}$ we have
\[
|X_\alpha(u, v,\ell, n)|=|b_j| \ n^{\sum_{s\in I_N} k_s (\beta_s-\alpha_s)}\underset{n\to\infty}{\to} 0
\]
thanks to  $\eqref{eq:c_j}$, $\eqref{eq:3.12}$ and
$\eqref{eq:k}$. So $\eqref{eq:19}$ and $\eqref{eq:20}$ follow.
\end{proof}

\begin{remark} \label{R:condition}
The assumption in Theorem~\ref{T:1} that $|\Phi(0)|<1$, which has only been used to establish Lemma~\ref{L:7-}, may be relaxed. Indeed, with a small modification in its proof we may replace in Lemma~\ref{L:7-} the assumption that $|\Phi(0)|<1$  by the following one: There exist two points $z_1$ and $z_2$ in $\Gamma_1$ and $\epsilon>0$ so that the interior of the triangle $\mbox{conv}\{ 0, -\epsilon z_1, -\epsilon z_2\}$ is contained in $\Phi^{-1}(\mathbb{D})$.   So Theorem~\ref{T:1} also holds under this weaker assumption and it applies for example to $\Phi(z)=\mbox{cos}(z)$, see \cite[Example 10]{BCP2}.
\end{remark}

\section{Proof of theorem~\ref{T:dos}}   \label{S:3}
We first establish the following lemma,  which we find of independent interest and state in a more general setting than that of Theorem~\ref{T:dos}.

\begin{lemma} \label{L:wM}
Let $T$ be a weakly mixing multiplicative operator on a commutative $F$-algebra $X$ over the real or complex scalar field $\K$, and let $N\ge 2$.  The following are equivalent.
\begin{enumerate}
\item[{\rm ($a$)}]\    For each $0\ne P\in \mathbb{K}[t_1,\dots ,t_N]$ with $P(0)=0$, the map $\widehat{P}:X^N\to X$, $f\mapsto P(f)$, has dense range.

\item[{\rm ($b$)}]\     For each $f\in HC(\underset{N}{\underbrace{T\oplus\dots\oplus T}})$ the algebra $A(f)$ is a hypercyclic algebra for $T$.
\end{enumerate}

\end{lemma}

\begin{proof}
$(a)\Rightarrow (b)$. Let $f=(f_1,\dots, f_N)\in X^N$ be hypercyclic for $T\oplus\dots \oplus T$, and let $0\ne g\in A(f)$. So $g=P(f)$ for some $0\ne P\in \mathbb{K}[t_1,\dots ,t_N]$ with $P(0)=0$. Given any open non-empty subset $V$ of $X$, we know by $(a)$ that $P^{-1}(V)$ is open and non-empty in $X^N$.  Letting $q\in \N$ so that $(T\oplus \dots \oplus T)^q f \in P^{-1}(V)$, by the multiplicativity of $T$ we have
\[
T^q(P(f))=P(T^q(f_1), \dots ,T^q(f_N))\in V.
\]
$(b)\Rightarrow (a)$ Fix a polynomial  $0\ne P\in \mathbb{K}[t_1,\dots ,t_N]$ with $P(0)=0$, and let $V\subset X$ be open and non-empty. Since $T$ is weakly mixing, by $(b)$ there exists $f\in X^N$ so that $P(f)$ is hypercyclic for $T$. So there exists  $q\in\N$ so that
\[
P(T^q(f_1),\dots ,T^q(f_N))=T^q(P(f))\in V.
\]
So $\widehat{P}$ has dense range. 
\end{proof}

\begin{proof}[Proof of Theorem~\ref{T:dos}]\ 

The implication $(d)\Rightarrow (a)$ is immediate, and the equivalence $(a)\Leftrightarrow (b)$ is \cite[Theorem 16]{BCP2}. To see $(a)\Rightarrow (c)$, let $N\ge 2$ be given and fix a non-zero polynomial $P\in \mathbb{K}[t_1,\dots ,t_N]$ with $P(0)=0$ and a non-empty open subset $V$ of $X$. Say,
\[
P(t) =\sum_{\alpha \in A} c_\alpha  t^\alpha  = \sum_{\alpha \in A} c_\alpha t_1^{\alpha_1}\dots t_N^{\alpha_N},
\]
where $A\subset \N_0^N$ is finite, non-empty, and not containing the zero vector
and where $0\ne c_\alpha \in \mathbb{K}$ for each $\alpha\in A$.  Let $m\geq 1$ be the degree of $P$ and write $P=\sum_{j=1}^m P_j$, where each polynomial $P_j$ is $j$-homogeneous. So

%and set $A_m=\{ \alpha\in A: \alpha_1+\dots + \alpha_N=m\}$. Then the $m$ homogeneous polynomial %$P_m$ of $P$ is 
\[
P_m(t)=\sum_{\alpha \in A_m} c_\alpha  t^\alpha  = \sum_{\alpha \in A_m} c_\alpha t_1^{\alpha_1}\dots t_N^{\alpha_N}
\]
where  $A_m=\{ \alpha\in A: |\alpha |=m\}$. Letting $\lambda=(\lambda_1,\dots, \lambda_N)\in \mathbb{K}^N$ so that 
$\sum_{\alpha \in A_m}c_{\alpha}\lambda^{\alpha}\ne 0$, we note that
\[
 p(z):=P(\lambda_1z,\dots ,\lambda_Nz)
 \]
 is a polynomial in $z$ of degree $m\ge 1$ and satisfies $p(0)=0$.  By $(a)$, there exists $f\in X$ generating a hypercyclic algebra, so there exists $q\in\N$ with $T^q(p(f))\in V$. Thus by the multiplicativity of $T$ we have 
\[
P(\lambda_1T^q(f),\dots, \lambda_N T^q(f))=p(T^q(f))=T^q(p(f))\in V.
\]
So $\widehat{P}$ has dense range.
Finally, we show $(c)\Rightarrow (d)$.  By $(c)$ and Lemma~\ref{L:wM} we know that for each $N\ge 2$ the set $H_N$ 
of $f\in X^N$ for which $A(f)$ is a hypercyclic algebra for $T$ is residual in $X^N$. Hence for each $N\ge 2$ the set
\[
G_N:=\{ g=(g_j)\in X^\N:\ (g_1,\dots ,g_N)\in H_N \} 
\]
is residual in $X^\N$, and hence so is the countable intersection 
\[
G:=\cap_{N\ge 2} G_N.
\]
Now, given $f=(f_j)\in G$ the algebra it generates may be written as
\[
A(f)=\cup_{N\ge 2} \{ P(f_1,\dots, f_N):\ 0\ne P\in\mathbb{K}[t_1,\dots, t_N]  \mbox{ with } P(0)=0 \}
\]
But given any $N\ge 2$ and $0\ne P\in\mathbb{K}[t_1,\dots, t_N]$ with $P(0)=0$, we know since $f\in G_N$ that $P(f_1,\dots,f_N)$ is hypercyclic for $T$. That is, $A(f)$ is a hypercyclic algebra and $(d)$ follows by Proposition~\ref{P:00}.% and Lemma~\ref{L:15}.
\end{proof}

\end{document}